\documentclass[leqno,a4paper]{article}
\usepackage{amssymb,amsmath,amsthm}
\usepackage[T1]{fontenc}

\title{\textbf{Quadratic Maps in Two Variables on Arbitrary Fields}}
\author{\textsc{R. Dur\'an D\'{\i}az$^1$}
\and
\textsc{L. Hern\'{a}ndez Encinas$^2$}
\and
\textsc{J. Mu\~{n}oz Masqu\'e$^2$}
\bigskip \\
$^1$ Departamento de Autom\'atica, Universidad de Alcal\'{a},\\
E-28871 Alcal\'{a} de Henares, Spain\\
E-mail: \texttt{raul.duran@uah.es}\\
$^2$ Instituto de Tecnolog\'{\i}as F\'{\i}sicas y de la Informaci\'on (ITEFI)\\
Consejo Superior de Investigaciones Cient\'{\i}ficas (CSIC),\\
E-28006 Madrid, Spain\\
E-mails: \texttt{\{luis, jaime\}@iec.csic.es}
}

\date{}

\newtheorem{theorem}{Theorem}[section]
\newtheorem{proposition}[theorem]{Proposition}
\newtheorem{lemma}[theorem]{Lemma}

\theoremstyle{remark}

\begin{document}

\maketitle

\begin{abstract}
\noindent Let $\mathbb{F}$ be a field of characteristic different
from $2$ and $3$, and let $V$ be a vector space of dimension $2$
over $\mathbb{F}$. The generic classification of homogeneous
quadratic maps $f\colon V\to V$ under the action of the linear
group of $V$, is given and efficient computational criteria to
recognize equivalence are provided.
\end{abstract}

{\small
\noindent \textit{Mathematics Subject Classification 2010:\/}
Primary 15A72; Secondary 11E88, 12E20, 12E30, 13A50, 15A66.
}

\medskip

{\small
\noindent\textit{Keywords:\/} Clifford algebra, homogeneous
quadratic map, invariant function, linear
representation, linear group, symmetric bilinear composition law.
}

\section{Introduction}
Let $\mathbb{F}$ be a field of characteristic $p\neq 2,3$ and let $V$ be a
$2$-dimensional $\mathbb{F}$-vector space. In this paper we classify homogeneous
quadratic maps $f\colon V\to V$ that satisfy certain generic condition to be
introduced later. Since there is a natural bijection between homogeneous
quadratic maps on $V$ and symmetric bilinear composition laws $F\colon V\times
V\to V$, the classification is carried out over the latter considering the
action of the general linear group $GL(V)$.

The topic has not elicited much attention and the literature is scarce. As far
as we know, only our previous work \cite{DMP} has clearly focused this topic.
However, homogeneous quadratic maps play a major role in the dynamics of
discrete systems (see, for example, \cite{DHM:2012:FSA}) and may give rise to
new or revamped one-way functions potentially interesting for cryptographic
applications.

Our purpose in this work is to apply to arbitrary fields the classification
obtained in \cite{DMP}, where such goal was achieved only for the case of an
an algebraically closed field. However, it will become apparent along the coming
sections that the methods employed to classify in the latter case are no longer
applicable. Actually, the main role for the present case is played by the
Clifford algebra associated to the quadratic form defining the symmetric
bilinear law. In particular, this new tool has allowed us to to deal with both
the hyperbolic and the elliptic cases of the quadratic form under a unified
framework. We want to stress that the new methods are totally different from
those used in \cite{DMP} so that the present contribution cannot be qualified as
a plain generalization or extension of the former one.

The main results of the paper can summarized as follows: We give both the general
form of any symmetric bilinear law such that its associated quadratic form
takes the value $-1$, and the explicit expression for the maps in the group of
isometries. It turns out that such maps are parametrized by an element of the
Clifford algebra whose Clifford norm is $1$. Next we compute the isotropy
group, which is discrete.
Last, in order to perform the classification, we resort to the
invariants computed in \cite{DMP} transposed to the case we are dealing with now.
While it is true that this technique do not provide a perfect classification,
we do supply efficient computational criteria, allowing one to recognize such
equivalence.

We stick here, as we did in \cite{DMP}, to the case of a 2-dimensional vector
space $V$. The reason will become clear, since the deployed techniques are
deeply connected to the 2-dimensional case. Apparently, each dimension claims
specific techniques and tools in order to achieve the classification. In a
sense, the procedure shows a kind of ``artistic'' flavor that renders it not
immediately or easily exportable to higher dimensions.

The paper is organized as follows: after a first section explaining some
preliminaries and notation, we focus on the topic of Clifford algebras, making
it apparent the role played by them in the present work; next we classify
generic symmetric bilinear laws, followed by the computation of the isotropy
group; finally we undertake the task of computing the criteria to recognize
the equivalence of symmetric bilinear laws.

\section{Preliminaries and notation}\label{intro_prelim}
If $(v_1,v_2)$ is a basis for $V$, then $f(x)=f_1(x)v_1+f_2(x)v_2$
where
\[
\begin{array}
[c]{l}
f_1(x_1,x_2)=a_1(x_1)^2+2b_1x_1x_2+c_1(x_2)^2,\\
f_2(x_1,x_2)=a_2(x_1)^2+2b_2x_1x_2+c_2(x_2)^2,
\end{array}
\quad
\begin{array}
[c]{l}
a_i,b_i,c_i\in\mathbb{F},1\leq i\leq 2,\\
x=x_1v_1+x_2v_2.
\end{array}
\]
As $p\neq 2$, there is a natural bijection between homogeneous quadratic
maps on $V$ and symmetric bilinear composition laws $F\colon V\times V\to V$,
$F(x,y)=x\star y$ (in short: SBLs), which is given by the polarization
formula, e.g., see \cite[XV, \S \S 2--3]{Lang}. Remember that two bilinear
laws $\star \colon (x,y)\in V^2\mapsto x\star y\in V$ and
$\circ \colon (x,y)\in (V^\prime)^2\mapsto x\circ y\in V^\prime $
are isomorphic---or $GL(V)$-equivalent---if and only if there is
a vector-space isomorphism
$u\colon V\overset{\simeq}{\longrightarrow}V^\prime $
such that,
$\forall (x,y)\in V^2$, $u(x)\circ u(y)=u(x\star y)$.

If $V,V^\prime $ are two $\mathbb{F}$-vector spaces,
the space of bilinear maps is denoted by
$L^2(V,V^\prime)$, with the natural identification
$L^2(V,V)\cong L(V\otimes V,V)\cong \otimes ^2V^\ast\otimes V$.
Hence, the classification problem that we
tackle transforms into a classification problem in the subspace
of symmetric tensors of type $(1,2)$ on the plane,
$S^2V^\ast \otimes V\subset \otimes^2V^\ast \otimes V$.
The natural action of the linear group $GL(V)$ on
$S^2V^\ast \otimes V$ is given by
\begin{equation}
\begin{array}
[c]{l}
\forall x,y\in V,F\in S^2V^\ast \otimes V,\;\forall u\in GL(V),\\
(u\cdot F)(x,y)=u\left( F\left( u^{-1}(x),u^{-1}(y)\right) \right) .
\end{array}
\label{f0}
\end{equation}
Let $(v_1^\ast ,v_2^\ast )$ be the dual basis of $(v_1,v_2)$; i.e.,
$v_i^\ast (v_{j})=\delta_{ij}$. Every $F\in S^2V^\ast \otimes V$ is
written as
\begin{align}
F & =v_1^\ast \otimes v_1^\ast \otimes\left( a_1v_1+a_2v_2\right)
+\left( v_1^\ast \otimes v_2^\ast +v_2^\ast \otimes v_1^\ast \right)
\otimes\left(  b_1v_1+b_2v_2\right)
\label{f3} \\
& +v_2^\ast \otimes v_2^\ast \otimes\left( c_1v_1+c_2v_2\right) ,
\nonumber
\end{align}
or matricially,
\[
\begin{array}
[c]{l}
F(x,y)=\left( \! (x_1,x_2)\left( \!
\begin{array}
[c]{ll}
a_1 & b_1\\
b_1 & c_1
\end{array}
\!\right) \left( \!
\begin{array}
[c]{l}
y_1\\
y_2
\end{array}
\right) \!,(x_1,x_2)\left( \!
\begin{array}
[c]{ll}
a_2 & b_2\\
b_2 & c_2
\end{array}
\!\right) \left( \!
\begin{array}
[c]{l}
y_1\\
y_2
\end{array}
\!\right) \!\right) ,\\
x=x_1v_1+x_2v_2,y=y_1v_1+y_2v_2.
\end{array}
\]
Let $\operatorname*{tr}\colon S^2V^\ast \otimes V\to V^\ast $
be the trace mapping. From \eqref{f3} we obtain
\begin{equation}
\operatorname*{tr}F=(a_1+b_2)v_1^\ast +(b_1+c_2)v_2^\ast .
\label{f6}
\end{equation}
The homomorphism
$F\in S^2V^\ast \otimes V\mapsto \operatorname*{tr}F\in V^\ast $
is proved to be $GL(V)$-equivariant. For a given $x\in V$, let
$F_x\colon V\to V$ be the $\mathbb{F}$-linear endomorphism
\begin{equation}
\forall y\in V,\quad F_x(y)=F(x,y).
\label{F_x}
\end{equation}
For each bilinear symmetric map $F\colon V\times V\to V$, let
$q_F\colon V\to \mathbb{F}$ be the quadratic form
defined by $q_F(x)=\det(F_x)$, where $F_x$ is the endomorphism
defined in \eqref{F_x}. As a computation shows,
\begin{align}
q_F(x)\!\!  &  =\!\!(x_1,x_2)\left(
\!
\begin{array}
[c]{cc}
a_1b_2-a_2b_1 & \!\frac{1}{2}(a_1c_2-a_2c_1)\\
\frac{1}{2}(a_1c_2-a_2c_1) & \!b_1c_2-b_2c_1
\end{array}
\!\right) \!\left(
\!
\begin{array}
[c]{l}
x_1\\
x_2
\end{array}
\! \right)
\label{f5}\\
\!\! & =\!\!\left( a_1b_2\!-\!a_2b_1\right) \left( x_1\right)
^2\!+\!\left( a_1c_2\! -\! a_2c_1\right) x_1x_2\!
+\! \left( b_1c_2\! -\! b_2c_1\right) \left( x_2\right) ^2,
\nonumber
\end{align}
$F$ being given as in \eqref{f3} and $x=x_1v_1+x_2v_2$.

\section{The Clifford algebra of $q_F$\label{Clifford_algebra}}
First of all, let $\star $ be a non-degenerate traceless SBL on $V$,
with associated symmetric bilinear map $F\colon V\times V\to V$,
and let $q=q_F$ be the quadratic form introduced
in the section \ref{intro_prelim}.

Given $x\in V$, the Cayley-Hamilton theorem yields
 $(F_x)^2=-q(x)\cdot \operatorname{id}V$, where
 $F_x$ is the endomorphism defined in \eqref{F_x}. Hence,
 by the universal property of the Clifford algebra (e.g.,
 see \cite[XIX, section 4]{Lang}), the linear map

\[
i\colon V\to \operatorname*{End}(V),\quad i(x)=F_x,
\quad
\forall x\in V,
\]
extends to a homomorphism
$\bar{\imath}\colon C(-q)\to \operatorname*{End}(V)$
from the Clifford algebra of $-q$ to $\operatorname*{End}(V)$.
Since $-q$ is non-degenerate with dimension $2$,
the algebra $C(-q)$ is central simple (e.g., see
\cite[Proposition 11.6--(1)]{EKM}); hence $\bar{\imath}$ is injective,
and since $\dim\operatorname*{End}(V)=4=\dim C(-q)$, we actually
conclude that $\bar{\imath}$ is an isomorphism of $\mathbb{F}$-algebras.

Moreover, $C_0(-q)$ is a quadratic $\mathbb{F}$-algebra and according to
\cite[Proposition 12.1]{EKM}, the form $-q$ represents $1$, i.e., $q$
represents $-1$.

Consequently, a non-degenerate quadratic form $Q\colon V\to \mathbb{F}$
is of the form $Q=q_F$ for some $F$ if and only if $Q$ takes the value $-1$.
If $\mathbb{F}$ is a finite field, then every quadratic form of rank $\geq 2$
on $\mathbb{F}$ takes any value of $\mathbb{F}^\ast $ (see
\cite[1.7. Proposition 4]{Serre}), but this does not necessarily happen
in an arbitrary field.

Let $x\mapsto\bar{x}$ be the conjugation in the Clifford algebra $C(-q)$;
it is the unique anti-automorphism of $C(-q)$ that restricts to $x\mapsto- x$
on $V$. As $q$ is a $2$-dimensional form, it is known that the map
$x\mapsto N(x)=x\cdot \bar{x}$ is multiplicative, maps into the ground field
$\mathbb{F}$ and extends $q$: This is the Clifford norm. Next, we choose
$v_1\in V$ such that $q(v_1)=-1$, leading to $N(v_1)=-1$. We consider
the linear isomorphism
$u\colon V\overset{\simeq}{\longrightarrow}C_0(-q)$,
$u(x)=v_1\cdot x$, $\forall x\in V$, which is actually an isometry from
$(V,q)$ to $(C_0(-q),N)$. A new SBL $\circ $ is defined on
$C_0(-q)$ as follows: $x\circ y=u\left( u^{-1}(x)\star u^{-1}(y)\right) $,
$\forall (x,y)\in C_0(-q)^2$. Hence, $(C_0(-q),\circ)$ is
isomorphic to $(V,F)$, and its associated quadratic form is $N$. As
$-q$ represents $1$, there exists an \emph{orthogonal} basis $(v_1,v_2)$
with respect to the symmetric bilinear form attached to $q$, such that,
\begin{equation}
\begin{array}
[c]{ll}
q(v_1)=-1, & q(v_2)=-\beta,
\end{array}
\quad\text{for some }\beta\in\mathbb{F}^\ast .
\label{basis}
\end{equation}
Accordingly,
$C(-q)=\left\langle 1,v_1,v_2,v_1\cdot v_2\right\rangle $
over $\mathbb{F}$, where the dot denotes the Clifford product,
and $(v_1\cdot v_2)\cdot(v_1\cdot v_2)=-\beta$.

If $x=x_0+x_1v_1+x_2v_2+x_{12}(v_1\cdot v_2)$,
$\bar{x}=x_0-x_1v_1-x_2v_2-x_{12}(v_1\cdot v_2)$, then,
$x\cdot\bar{x}=(x_0)^2-(x_1)^2-\beta(x_2)^2+\beta(x_{12})^2$.

\section{Classification of SBLs}
\begin{proposition}
\label{propos1}
Every SBL on $C_0(-q)$ can be written
in the following form:
\begin{equation}
F_{abc}(x,y)=a\cdot x\cdot y
+b\cdot(\bar{x}\cdot y+x\cdot\bar{y})
+c\cdot\overline{x\cdot y},
\qquad
\forall x,y\in C_0(-q),
\label{F_abc}
\end{equation}
for some $(a,b,c)\in C_0(-q)^3$.
\end{proposition}
\begin{proof}
We have
$C_0(-q)=\{ x=x_0+x_{12}(v_1\cdot v_2):x_0,x_{12}\in \mathbb{F}\} $.
Hence the elements of degree zero of the Clifford algebra admit
the basis $\{1,v_1\cdot v_2\}$.

The mappings \eqref{F_abc} are obviously $\mathbb{F}$-bilinear
and symmetric, since the Clifford product is $\mathbb{F}$-bilinear
and the conjugation $x\mapsto \bar{x}$ is an $\mathbb{F}$-linear
anti-automorphism. Letting
$a=a_0+a_{12}(v_1\cdot v_2)$, $b=b_0+b_{12}(v_1\cdot v_2) $,
$c=c_0+c_{12}(v_1\cdot v_2)$, it follows that the mappings $F_{abc}$
depend on the $6$ parameters $a_0$, $a_{12}$, $b_0$, $b_{12}$, $c_0$,
$c_{12}$. As $\dim(S^2V^\ast \otimes V)=6$, we can conclude.
\end{proof}

The equations of the isomorphism
$u\colon V\overset{\simeq}{\longrightarrow}C_0(-q)$,
$u(x)=v_1\cdot x$, where $x=x_1v_1+x_2v_2\in V$
(introduced in the section \ref{Clifford_algebra})
and those of its inverse are the following:
\[
u(x)=x_1+x_2(v_1\cdot v_2),\quad u^{-1}
\left( x_0+x_{12}(v_1\cdot v_2)\right)
=x_0v_1+x_{12}v_2.
\]
In what follows, we shall identify the mappings $F_{abc}$ and
$F=u^{-1}\circ F_{abc}\circ (u,u)$. As a computation shows,
we have the following formulas:

\begin{equation}
\begin{array}
[c]{rl}
\left(
\begin{array}
[c]{cc}
a_1 & b_1\\
b_1 & c_1
\end{array}
\right) = & \left(
\begin{array}
[c]{cc}
a_0+2b_0+c_0
& -\beta \left( a_{12}-c_{12}\right) \\
-\beta \left(  a_{12}-c_{12}\right)
& -\beta \left( a_0-2b_0+c_0\right)
\end{array}
\right) ,
\smallskip\\
\left(
\begin{array}
[c]{cc}
a_2 & b_2\\
b_2 & c_2
\end{array}
\right)
= & \left(
\begin{array}
[c]{cc}
a_{12}+2b_{12}+c_{12} & a_0-c_0\\
a_0-c_0 & -\beta
\left(  a_{12}-2b_{12}+c_{12}\right)
\end{array}
\right) .
\end{array}
\label{ABC}
\end{equation}
\begin{theorem}
\label{th1}
The SBLs on $C_0(-q)$ with attached quadratic form $N$
are the maps of the form
$F_{a,c}(x,y)=a\cdot x\cdot y+c\cdot \overline{x\cdot y}$,
for some $(a,c)\in C_0(-q)\times C_0(-q)$ with $N(c)-N(a)=1$.

If $G=\{ \lambda \in C_0(-q):N(\lambda)=1\}$, then two mappings
$F_{ac}$ and $F_{a^\prime c^\prime }$ are isomorphic
if and only if there exists $\lambda \in G$ such that
$(a^\prime ,c^\prime )=(\lambda ^{-1}a,\lambda ^3c)$ or
$(a^\prime ,c^\prime )=(\lambda ^{-1}\bar{a},\lambda ^3\bar{c})$.
\end{theorem}
\begin{proof}
According to \eqref{ABC} we have
\[
\begin{array}
[c]{lll}
a_1=a_0+2b_0+c_0, & b_1
=-\beta \left( a_{12}-c_{12}\right) ,
& c_1=-\beta \left( a_0-2b_0+c_0
\right) ,\\
a_2=a_{12}+2b_{12}+c_{12}, & b_2=a_0-c_0, & c_2
=-\beta \left(
a_{12}-2b_{12}+c_{12}
\right) .
\end{array}
\]
By replacing these formulas into \eqref{f5}, it follows:
\[
\begin{array}
[c]{rl}
eq_1\equiv & a_1b_2-a_2b_1\\
= & (a_0)^2+\beta (a_{12})^2+2a_0b_0+2\beta a_{12}b_{12}-2b_0
c_0-2\beta b_{12}c_{12}\\
& -(c_0)^2-\beta(c_{12})^2,\medskip\\
eq_2\equiv & \tfrac{1}{2}(a_1c_2-a_2c_1)\\
= & 2\beta a_0b_{12}-2\beta a_{12}b_0+2\beta b_{12}c_0
-2\beta b_0c_{12},
\medskip\\
eq_3\equiv & b_1c_2-b_2c_1\\
= & \beta(a_0)^2+\beta^2(a_{12})^2-2\beta a_0b_0
-2\beta ^2a_{12}b_{12}+2\beta^2b_{12}c_{12}+2\beta b_0c_0\\
& -\beta (c_0)^2-\beta ^2(c_{12})^2.
\end{array}
\]
Hence
\begin{equation}
eq_1=-1,\quad eq_2=0,\quad eq_3=-\beta .
\label{system0}
\end{equation}
Dividing $eq_2=0$ by $2\beta $ and $eq_3=-\beta $ by $\beta $ we obtain
{\small
\[
\begin{array}
[c]{r}
(a_0)^2+\beta (a_{12})^2+2a_0b_0+2\beta a_{12}b_{12}
-2b_0c_0-2\beta b_{12}c_{12}-(c_0)^2-\beta (c_{12})^2=-1,\\
a_0b_{12}-a_{12}b_0+b_{12}c_0-b_0c_{12}=0,\\
(a_0)^2+\beta(a_{12})^2-2a_0b_0-2\beta a_{12}b_{12}
+2\beta b_{12}c_{12}+2b_0c_0-(c_0)^2-\beta (c_{12})^2=-1.
\end{array}
\]
}
By adding and subtracting the first and third equations above
and dividing the result by $2$,
\begin{align*}
a_0^2+\beta a_{12}^2-c_0^2-\beta c_{12}^2+1 & =0,\\
a_0b_0-b_0c_0+\beta a_{12}b_{12}-\beta b_{12}c_{12} & =0.
\end{align*}
Accordingly, the system \eqref{system0} is equivalent to
\[
\begin{array}
[c]{rl}
e_1\equiv & (a_0)^2+\beta (a_{12})^2-(c_0)^2-\beta (c_{12}
)^2+1=0,\\
e_2\equiv & a_0b_{12}-a_{12}b_0+b_{12}c_0-b_0c_{12}=0,\\
e_3\equiv & a_0b_0-b_0c_0+\beta a_{12}b_{12}-\beta b_{12}c_{12}=0,
\end{array}
\]
which we use in what follows, because it is easier than
the first one. The equation $e_1$ can equivalently be written as
\begin{equation}
N(c)-N(a)=1. \label{N_c_N_a}
\end{equation}
Furthermore, the equations $e_2=e_3=0$ are linear in $b_0$ and $b_{12}$
and they can be written in matrix notation as
\begin{equation}
\left(
\begin{array}
[c]{cc}
-a_{12}-c_{12} & a_0+c_0 \\
a_0-c_0 & \beta (a_{12}-c_{12})
\end{array}
\right) \left(
\begin{array}
[c]{c}
b_0\\
b_{12}
\end{array}
\right) =\left(
\begin{array}
[c]{c}
0\\
0
\end{array}
\right) .
\label{system1}
\end{equation}
The determinant of the matrix of the system \eqref{system1}
is equal to
\[
e_4\equiv(c_0)^2+\beta(c_{12})^2-(a_0)^2-\beta (a_{12})^2
=N(c)-N(a)=1,
\]
by virtue of \eqref{N_c_N_a}; hence $e_4$ cannot vanish.
Therefore, $b_0=b_{12}=0$.

Finally, let us determine the conditions under which $F_{a,c}$
and $F_{a^\prime ,c^\prime }$ are isomorphic. As is known,
any isomorphism between them must
be an isometry of $(C_0(-q),N)$, and these isometries
are the group $\mathcal{G}$ of the maps that have one
of the following forms:
\begin{equation}
\left.
\begin{array}
[c]{rl}
\text{(i)} & x\mapsto \lambda x,\\
\text{(ii)} & x\mapsto \lambda\bar{x},
\end{array}
\right\}  \quad\forall x\in C_0(-q),\forall\;\lambda \in G.
\label{group}
\end{equation}
Letting
$x^\prime =\lambda x$, $y^\prime =\lambda y$, $z^\prime =\lambda z$
into the equation
\begin{equation}
z^\prime =F_{a^\prime c^\prime }(x^\prime ,y^\prime )
=a^\prime x^\prime y^\prime +c^\prime \bar{x}^\prime \bar{y}^\prime ,
\label{Fprime}
\end{equation}
we obtain $\lambda z=a^\prime (\lambda x)(\lambda y)
+c^\prime (\bar{\lambda }\bar{x})(\bar{\lambda}\bar{y})$; hence
$z=\lambda a^\prime xy+\bar{\lambda }^3c^\prime \bar{x}\bar{y}$, as
$\lambda ^{-1}=\bar{\lambda }$ (because $N(\lambda )=\lambda \bar{\lambda }=1$)
and comparing it with the original equation, i.e.,
$z=F_{ac}(x,y)=axy+c\bar{x}\bar{y}$, we deduce $a^\prime =\lambda ^{-1}a$,
$c^\prime =\lambda ^3c$. Similarly, letting
$x^\prime =\lambda\bar{x}$, $y^\prime =\lambda \bar{y}$,
$z^\prime =\lambda \bar{z}$ into \eqref{Fprime}, we obtain
$\lambda \bar{z}=a^\prime (\lambda\bar{x})(\lambda\bar{y})
+c^\prime (\bar{\lambda}x)(\bar{\lambda}y)$; hence
$\bar{z}=\lambda a^\prime \bar{x}\bar{y}+\bar{\lambda }^3c^\prime xy$
and conjugating, $z=\bar{\lambda }\bar{a}^\prime xy
+\lambda ^3\bar{c}^\prime \bar{x}\bar{y}$. Therefore it follows:
$a=\bar{\lambda }\bar{a}^\prime $,
$c=\lambda ^3\bar{c}^\prime $, or equivalently, $a^\prime =\lambda ^{-1}
\bar{a}$, $c^\prime =\lambda ^3\bar{c}$, thus concluding the proof.
\end{proof}

\section{Isotropy}
Next, we discuss the index of $q$. The quadratic form $q$ is said to be
\emph{hyperbolic} if $q$ admits an isotropic vector $v_1\neq 0$. If $q$ does
not admit any non-zero isotropic vector, then $q$ is said to be
\emph{elliptic}; in this case, as we have seen above, there exists a basis
$(v_1,v_2)$ for $V$ such that, $q(x)=-(x_1)^2-\beta(x_2)^2$, where
$-\beta\notin\mathbb{F}^{\ast2}$.

If the discriminant of $q$ is different from $1\operatorname{mod}
\mathbb{F}^{\ast 2}$, then $q$ is elliptic, and if the discriminant of $q$ is
equal to $1\operatorname{mod}\mathbb{F}^{\ast 2}$, then $q$ is hyperbolic.

With the same notations as in the section \ref{Clifford_algebra}, we have
$(v_1\cdot v_2)^2+\beta =0$. The $\mathbb{F}$-algebra $C_0(-q)$ being
quadratic, we deduce $C_0(-q)\cong\mathbb{F}[t]/(t^2+\beta )$ (e.g., see
\cite[Example 98.2]{EKM}). Hence, in the hyperbolic case,
$C_0(-q)\cong \mathbb{F}\times \mathbb{F}$, and in the elliptic case
$C_0(-q)$ is a quadratic field extension of the ground field $\mathbb{F}$.

\begin{lemma}
\label{lemma3}
If $q$ is hyperbolic, then the set of zero divisors in
$C_0(-q)$ coincides with the set of elements of norm zero.
\end{lemma}

\begin{proof}
If $N(x)=x\cdot \bar{x}=0$, then $x$ is a zero divisor obviously. Conversely,
if $x,y\in C_0(-q)$ are such that $x\neq 0$, $y\neq 0$, and $x\cdot y=0$, then
we obtain the following homogeneous linear system:
\begin{equation}
\left(
\begin{array}
[c]{cc}
x_0 & -\beta x_{12}\\
x_{12} & x_0
\end{array}
\right) \left(
\begin{array}
[c]{c}
y_{0}\\
y_{12}
\end{array}
\right) =\left(
\begin{array}
[c]{c}
0\\
0
\end{array}
\right) . \label{system2}
\end{equation}
Since
\[
\det \left(
\begin{array}
[c]{cc}
x_{0} & -\beta x_{12}\\
x_{12} & x_0
\end{array}
\right) =N(x),\quad y\neq 0,
\]
we can conclude the statement.
\end{proof}

Below we compute the isotropy subgroup $\mathcal{G}(F_{ac})\subset \mathcal{G}$
of the mapping $F_{ac}$ in Theorem \ref{th1}.

We denote by $\phi _\lambda (x)=\lambda x$, $\psi _\lambda (x)=\lambda \bar{x}$,
$\lambda \in G$, $x\in C_0(-q)$, the transformations (i) and (ii)
respectively in the formula \eqref{group}. As a computation shows, we obtain
\[
\begin{array}
[c]{cccc}
\phi _\lambda \circ \phi _\mu =\phi _{\lambda \mu },
& \phi _\lambda \circ \psi _\mu =\psi _{\lambda \mu},
& \psi_{\lambda}\circ\psi_{\mu}=\phi _{\lambda \bar{\mu}},
& \forall \lambda ,\mu \in G.
\end{array}
\]
In particular
$\psi _\lambda\circ \psi _\lambda =\phi _{\lambda \bar{\lambda }}
=\operatorname*{id}$, $\forall \lambda \in G$; i.e.,
every transformation in (ii) is involutive.

\begin{proposition}
\label{propos2}
With the previous notations, we have

If $N(a)\neq 0$ and $ca^3\notin\mathbb{F}$, then
$\mathcal{G}(F_{ac})=\{ \operatorname*{id}\} $.

If $N(a)\neq 0$ and $ca^3\in \mathbb{F}$, then
$\mathcal{G}(F_{ac})=\left\{
\operatorname*{id},\psi _{\frac{\bar{a}}{a}}\right\} $.

If $N(a)=0$, then
$\mathcal{G}(F_{ac})
=\left\{ \phi _\lambda ,\psi _\mu :\lambda ^3=1,\mu ^3=c^2\right\} $.
\end{proposition}

\begin{proof}
If one of the transformations (i) or (ii) in the formula \eqref{group} belongs
to $\mathcal{G}(F_{ac})$, then either (i) $\lambda a=a$, $c=\lambda ^3c$, or
(ii) $\lambda a=\bar{a}$, $c=\lambda ^3\bar{c}$. We distinguish several cases.

Assume the item (i) holds.

\begin{enumerate}
\item If $q$ is elliptic, then $a\neq 0$ implies $\lambda =1$, as $C_0(-q)$ is
a field, and $a=0$ implies $\lambda ^3=1$, because in this case $c$ is
invertible, as follows from \eqref{N_c_N_a}.

\item If $q$ is hyperbolic, i.e., $\beta =-\gamma ^2$,
$\gamma\in \mathbb{F}^\ast $, then by applying Lemma \ref{lemma3} to the equation
$(\lambda -1)a=0$ it follows that either $\lambda =1$ or $N(a)=0$. In the second
case $c$ is invertible in $C_0(-q)$ by virtue of \eqref{N_c_N_a}; hence
$\lambda ^3=1$. If $N(a)\neq 0$, then the equation $(\lambda -1)a=0$ implies
$\lambda =1$.
\end{enumerate}

If $q$ is elliptic, then $N(a)=0$ if and only if $a=0$. Therefore, we can
group the two previous items saying that the transformations of type (i) in
\eqref{N_c_N_a} that belong to $\mathcal{G}(F_{ac})$ are as follows: If
$N(a)\neq 0$, then such transformations reduce to the identity map, and if
$N(a)=0$, then they correspond to the values $\lambda \in G$ such that
$\lambda ^3=1$.

Assume the item (ii) holds.

From $\lambda a=\bar{a}$ it follows $\lambda a^2=N(a)$.

If $N(a)=0$, then $a=0$, as $\lambda $ is invertible and $C_0(-q)$\ has no
nilpotent element. In this case $N(c)=1$ and from $c=\lambda^3\bar{c}$ it
follows $\lambda ^3=c^2$.

\begin{itemize}
\item If $q$ is elliptic, the equation $\lambda ^3=c^2$ may admit none (if
$c^2$ is not a cube in $C_0(-q)$), one (if $c^2$ is a cube in
$C_0(-q)$ and $-3$ is not a square in $C_0(-q)$) or three solutions in
$C_0(-q)$ (if $c^2$ is a cube in $C_0(-q)$ and $-3$ is a square in
$C_0(-q)$).

\item If $q$ is hyperbolic, then by considering the isomorphism
\begin{equation}
\begin{array}
[c]{l}
\phi \colon\mathbb{F}[t]/(t^2-\gamma ^2)
\to \mathbb{F}\times \mathbb{F},\\
\phi (u+v\tau)=(u-v\gamma ,u+v\gamma ),\\
u,v\in \mathbb{F},\;\tau =t\operatorname{mod}(t^2-\gamma ^2),
\end{array}
\label{iso}
\end{equation}
and by writing $\phi (w)=(w_1,w_2)$, it follows that the equation
$\lambda ^3=c^2$ is equivalent to the pair of equations
$(\lambda _1)^3=(c_1)^2$, $(\lambda _2)^3=(c_2)^2$ in $\mathbb{F}$. As
$N(c)=N(\lambda )=1$, we have $\lambda _1\lambda _2=c_1c_2=1$, and the
equation $(\lambda _2)^3=(c_2)^2$ is equivalent to
$(\lambda _1)^3=(c_1)^2$. Hence even in the hyperbolic case
the number of solutions to $\lambda ^3=c^2$ may be $0$, $1$ or $3$.
\end{itemize}

If $N(a)\neq 0$, then $\lambda =\frac{\bar{a}}{a}$ and replacing this value into
the second equation in (ii) we obtain $a^3c=\overline{a^3c}$, or
equivalently $a^3c\in \mathbb{F}$.

In summary, the transformations of type (ii) in \eqref{N_c_N_a} that belong to
$\mathcal{G}(F_{ac})$ are as follows:

\begin{itemize}
\item If $N(a)=0$, then such transformations correspond to the values
$\lambda\in G$ such that $\lambda^3=c^2$, whether $q$ is elliptic or hyperbolic.

\item If $N(a)\neq 0$, then such transformations do not exist, except when
$ca^3\in \mathbb{F}$, in which case the only transformation of type (ii) in
$\mathcal{G}(F_{ac})$ corresponds to $\lambda =\frac{\bar{a}}{a}$.
\end{itemize}

Accordingly, we have

\[
\text{(i)}\left\{
\begin{array}
[c]{ll}
N(a)\neq 0, & \{\operatorname*{id}\}\\
N(a)=0, & \{\lambda \in G:\lambda ^3=1\}
\end{array}
\right. \quad\text{(ii)}\left\{
\begin{array}
[c]{l}
N(a)\neq 0\left\{
\begin{array}
[c]{lc}
\emptyset, & \text{if }ca^3\notin \mathbb{F}\\
\lambda=\frac{\bar{a}}{a}, & \text{if }ca^3\in\mathbb{F}
\end{array}
\right. \\
\begin{array}
[c]{cc}
N(a)=0, & \{ \lambda \in G:\lambda ^3=c^2\}
\end{array}
\end{array}
\right.
\]
Putting together transformations of type (i) and type (ii), the statement follows.
\end{proof}

\section{The role of the invariants}
Let $\sigma\colon V^\ast \to S^2V^\ast \otimes V$ be the map
defined by,
\begin{equation}
\sigma(v^\ast )(x,y)
=\tfrac{1}{3}\left(  v^\ast (x)y+v^\ast (y)x\right) ,
\quad x,y\in V,\;v^\ast \in V^\ast . \label{f2}
\end{equation}
By using formula \eqref{f0}, the homomorphism $\sigma $ is proved to be a
$GL(V)$-equivariant section of $\mathrm{tr}$. If
$v^\ast =\lambda _1v_1^\ast +\lambda _2v_2^\ast$, then from \eqref{f2} it follows:
\begin{equation}
\begin{array}
[c]{l}
\sigma (v^\ast )(v_1,v_1)=\tfrac{2}{3}\lambda _1v_1,\smallskip \\
\sigma (v^\ast )(v_2,v_2)=\tfrac{2}{3}\lambda _2v_2,\smallskip \\
\sigma (v^\ast )(v_1,v_2)=\tfrac{1}{3}(\lambda _1v_2+\lambda _2v_1).
\end{array}
\label{sigma}
\end{equation}
Therefore, there is a decomposition of $GL(V)$-modules $S^2V^\ast \otimes
V=W\oplus\sigma(V^\ast )$, where $W=\left\{  F\in S^2V^\ast \otimes
V:\mathrm{tr}F=0\right\}  $.

For every $F\in S^2V^\ast \otimes V$ we set $\bar{F}=F-\sigma
(\mathrm{tr}F)$. Then, $F$ is said to be \emph{regular} if the quadratic form
$Q_{\bar{F}}$ is non-degenerate.

A simple computation proves that $F$ is regular if and only if the following
condition holds:

\begin{align}
\det Q_{\bar{F}}  &  =\tfrac{4}{27}a_1b_1b_2c_2 -\tfrac{1}{3}
a_1a_2b_1c_1+\tfrac{2}{3}a_2b_1b_2c_1 +\tfrac{1}{6}a_1
a_2c_1c_2\label{f7}\\
&  -\tfrac{1}{3}a_2b_2c_1c_2-\tfrac{1}{27}(a_1)^3c_1
+\tfrac{1}{27}(a_1)^2(b_1)^2+\tfrac{1}{108}(a_1)^2(c_2)^2\nonumber \\
&  +\tfrac{8}{27}(b_2)^3c_1+\tfrac{4}{27}(b_1)^2(b_2)^2
+\tfrac{1}{27}(b_2)^2(c_2)^2+\tfrac{8}{27}a_2(b_1)^3\nonumber \\
&  -\tfrac{1}{27}a_2(c_2)^3-\tfrac{4}{27}b_1(b_2)^2c_2
-\tfrac{1}{27}(a_1)^2b_1c_2-\tfrac{4}{9}a(b_2)^2c_1\nonumber \\
&  +\tfrac{2}{9}(a_1)^2b_2c_1+\tfrac{2}{9}a_2b_1(c_2)^2
-\tfrac{4}{9}a_2(b_1)^2c_2-\tfrac{1}{27}a_1b_2(c_2
)^2\nonumber\\
&  -\tfrac{4}{27}a_1(b_1)^2b_2-\tfrac{1}{4}(a_2)^2(c_1)^2\nonumber \\
&  \neq 0.\nonumber
\end{align}
From the very definition it follows that the set of regular bilinear symmetric
maps is an open subset $R\subset S^2V^\ast \otimes V$ in the Zariski
topology; precisely, the set where the quartic form \eqref{f7} does not vanish.

If the ground field $\mathbb{F}$ is algebraically closed, then in
\cite[Theorem 4--2]{DMP} it is proved that two regular elements $F,G\in
R\subset S^2V^\ast \otimes V$ are $GL(V)$-equivalent, if and only if
$\mathcal{I}_i(F)=\mathcal{I}_i(G)$, $i=1,2$, where
$\mathcal{I}_1,\mathcal{I}_2\colon R\to \mathbb{F}$ are the $GL(V)$-invariant
functions defined in \cite[Theorem 4--1]{DMP} and computed in
\cite[pp.\ 11--12]{DMP}), namely,
\[
\begin{array}
[c]{rl}
\mathcal{I}_1(F)= & \tfrac{1}{12\det Q_{\bar{F}}}
\left[  (a_1+b_2)^2\left(  (2b_1-c_2)^2+3(2b_2-a_1)c_1\right) \right. \\
& \multicolumn{1}{r}{+(a_1+b_2)(b_1+c_2)((2b_2-a_1)(2b_1-c_2)-9a_2c_1)}\\
& \multicolumn{1}{r}{\left.  +(b_1+c_2)^2((2b_2-a_1)^2+3(2b_1-c_2)a_2)\right] ,}
\end{array}
\]
\[
\begin{array}
[c]{rl}
\mathcal{I}_2(F)= & \tfrac{1}{4\det Q_{\bar{F}}}
\left[  -c_1(a_1+b_2)^3+(a_1+b_2)^2(b_1+c_2)(2b_1-c_2)\right. \\
& \multicolumn{1}{r}{\left. +(a_1+b_2)(b_1+c_2)^2(2b_2-a_1)-a_2(b_1+c_2)^3\right] ,}
\end{array}
\]
where $\det Q_{\bar{F}}$ is as in \eqref{f7}. Next, we discuss the role of
these invariants in the classification problem.

The isomorphism $u\colon V\to C_0(-q)$, $u(x)=v_1\cdot x$, $\forall x\in V$,
defined in the section \ref{Clifford_algebra}, induces an isomorphism
\[
S^2(u^{\ast-1})\otimes u\colon S^2(V^\ast ) \otimes V\to S^2
(C_0(-q)^\ast )\otimes C_0(-q)
\]
which allows one to transport the invariants $\mathcal{I}_1$ and
$\mathcal{I}_2$ into a Zariski open subset $R^\prime $ in $S^2
(C_0(-q)^\ast )\otimes C_0(-q)$. Moreover, by applying Theorem \ref{th1},
we can confine ourselves to compute these new invariants $\mathcal{J}_1$ and
$\mathcal{J}_2$ only on the maps $F_{ac}$ fulfilling the equation
\eqref{N_c_N_a}. This is accomplished by using the formulas \eqref{ABC} for
$b_0=b_{12}=0$, thus obtaining the following expressions:

\begin{equation}
\begin{array}
[c]{rrl}
\text{\lbrack1]} & \mathcal{J}_1\left(  F_{ac}\right)
= & 27\frac{K(a,c)+2N(a)^2+3N(a)}{4K(a,c)+8N(a)^2+36N(a)+27},
\medskip\\
\text{\lbrack2]} & \mathcal{J}_2\left(  F_{ac}\right)
= & 27\frac{K(a,c)+2N(a)^2}{4K(a,c)+8N(a)^2+36N(a)+27},
\medskip\\
& K(a,c)= & a^3c+\bar{a}^3\bar{c},
\end{array}
\label{invariants}
\end{equation}

\begin{proposition}
\label{propos2.5}
Let $F_{ac}$, $F_{a^\prime c^\prime }$ be two symmetric
bilinear maps in Zariski open subset $R^\prime $ defined above.

If the pairs $(a,c)$ and $(a^\prime ,c^\prime )$ are related by one of the
two formulas in the second part of \emph{Theorem \ref{th1}}, then
$\mathcal{J}_i\left( F_{ac}\right)
=\mathcal{J}_i\left(  F_{a^\prime c^\prime }\right)  $, $i=1,2$.
Therefore, the functions $\mathcal{J}_1$ and $\mathcal{J}_2$
are invariant under the action of the group of
transformations $\mathcal{G}$ given in \emph{\eqref{group}}.

Moreover, the equations $\mathcal{J}_i\left( F_{ac}\right)
=\mathcal{J}_i\left(  F_{a^\prime c^\prime }\right)  $, $i=1,2$,
hold if and only if the following two conditions are fulfilled:
\begin{equation}
\begin{array}
[c]{lllll}
\text{\emph{(i)}}
& N(a^\prime )=N(a),
&  &
\text{\emph{(ii)}} & K(a^\prime ,c^\prime )=K(a,c).
\end{array}
\label{conditions}
\end{equation}

In addition, we have the following mutually excluding cases:

\begin{enumerate}
\item[$(1)$] If $a=0$ and \emph{(i)} and \emph{(ii)} in
\emph{\eqref{conditions}} hold, then $a^\prime =0$, and the symmetric
bilinear maps $F_{ac}$ and $F_{a^\prime c^\prime }$ are
$\mathcal{G}$-equivalent if and only if $\frac{c^\prime }{c}$
or $\frac{c^\prime }{\bar{c}}$\ is a cube in $C_0(-q)$.

\item[$(2)$] If $c=0$, and \emph{(i)} and \emph{(ii)} in
\emph{\eqref{conditions}} hold, then $c^\prime =0$, and the symmetric
bilinear maps $F_{ac}$ and $F_{a^\prime c^\prime }$ are equivalent under the
subgroup $\mathcal{G}^{0}\subset\mathcal{G}$ of the transformations of type
\emph{(i)} in \emph{\eqref{group}}.

\item[$(3)$] If $a\neq 0$ and $c\neq 0$, then the formulas \emph{(i)} and
\emph{(ii)} in \emph{\eqref{conditions}} hold if and only if the symmetric
bilinear maps $F_{ac}$ and $F_{a^\prime c^\prime }$ are $\mathcal{G}$-equivalent.
\end{enumerate}
\end{proposition}

\begin{proof}
If $a^\prime =\lambda^{-1}a,c^\prime =\lambda^3c$ or
$a^\prime =\lambda^{-1}\bar{a},c^\prime =\lambda^3\bar{c}$, then taking account of
the fact that $\lambda ^{-1}=\bar{\lambda }$ as $\lambda \in G$, a
straightforward computation shows that $\mathcal{J}_i\left( F_{ac}\right)
=\mathcal{J}_i\left( F_{a^\prime c^\prime }\right) $, $i=1,2$.

Moreover, solving the equations [1] and [2] in \eqref{invariants} with respect
to $K(a,c)$ and $N(a)$ it follows
\[
\begin{array}
[c]{rl}
K(a,c)= & -27\frac{6\mathcal{J}_1\left( F_{ac}\right) ^2-2\mathcal{J}_2
\left( F_{ac}\right)  ^2-27\mathcal{J}_2\left( F_{ac}\right)
}{\left[ 12\mathcal{J}_1\left( F_{ac}\right) -8\mathcal{J}_2\left(
F_{ac}\right) -27\right] ^2},\medskip\\
N(a)= & 9\frac{\mathcal{J}_2\left(  F_{ac}\right) -\mathcal{J}_1
\left(
F_{ac}\right) }{12\mathcal{J}_1\left( F_{ac}\right) -8\mathcal{J}_2
\left(  F_{ac}\right) -27}.
\end{array}
\]

Hence the equations $\mathcal{J}_i\left( F_{ac}\right)
=\mathcal{J}_i\left( F_{a^\prime c^\prime }\right) $, $i=1,2$, imply
$K(a,c)=K(a^\prime ,c^\prime )$ and $N(a)=N(a^\prime )$.

\medskip

$(1)$ From \eqref{conditions}-(i) it follows $N(a^\prime )=0$, and by virtue
of \eqref{N_c_N_a}, we conclude that $N(c)=N(c^\prime )=1$.

If $q$ is elliptic, this implies $a^\prime =0$, and $c$ and $c^\prime $ are
invertible in $C_0(-q)$. If $F_{ac}$ and $F_{a^\prime c^\prime }$ are
$\mathcal{G}$-equivalent, then $\frac{c^\prime }{c}$ or
$\frac{c^\prime }{\bar{c}}$ belong to the group $G$ defined in Theorem \ref{th1};
the converse is obvious.

If $q$ is hyperbolic, we can apply the isomorphism \eqref{iso}; by using the
notations introduced therein, the formulas (i) and (ii) in \eqref{conditions}
transform respectively into: (i') $a_1^\prime a_2^\prime =0$, (ii')
$(a_1^\prime )^3c_1^\prime +(a_2^\prime )^3c_2^\prime =0$, and
$N(c^\prime )=1$ means (iii') $c_1^\prime c_2^\prime =1$. If
$a_1^\prime =a_2^\prime =0$, (i.e., $a^\prime =0$), then (ii') holds
identically and we can conclude as in the previous case. If, for example, we
had $a_1^\prime \neq 0$, $a_2^\prime =0$, then (ii')
implies $(a_1^\prime )^3c_1^\prime =0$, and since
$c_1^\prime \in \mathbb{F}^\ast $ it follows $a_1^\prime =0$,
thus leading us to a contradiction.

\medskip

$(2)$ From \eqref{N_c_N_a} and \eqref{conditions}-(i) it follows
$N(a)=N(a^\prime )=-1$, $N(c)=N(c^\prime )=0$.

If $q$ is elliptic, this implies $c=c^\prime =0$, and $a$, $a^\prime $ are
invertible in $C_0(-q)$ and $\lambda=\frac{a}{a^\prime }$ belongs to $G.$

If $q$ is hyperbolic, then by using the isomorphism \eqref{iso}, the equation
(ii) in \eqref{conditions} transforms into
(ii') $(a_1^\prime )^3c_1^\prime +(a_2^\prime )^3c_2^\prime =0$,
and furthermore we have
$a_1a_2=a_1^\prime a_2^\prime =-1$, $c_1^\prime c_2^\prime =0$.
If $c_1^\prime =0$, then (ii') becomes $(a_2^\prime )^3c_2^\prime =0$,
and since $a_2^\prime $ is invertible we deduce that
$c_2^\prime =0$; similarly, $c_2^\prime =0$ implies $c_1^\prime =0$.
Hence $c^\prime =0$, in which case we have
$\lambda =\frac{a}{a^\prime }\in G$.

\medskip

$(3)$ If $q$ is elliptic, then $C_0(-q)$ is a field and by virtue of the
assumption it follows that the elements $a$, $c$, $\bar{a}$, $\bar{c}$,
$a^\prime $, $c^\prime $, $\bar{a}^\prime $, and $\bar{c}^\prime $ are
invertible. Letting $a^\prime =\frac{a\bar{a}}{\bar{a}^\prime }$,
$c^\prime =\frac{c\bar{c}}{\bar{c}^\prime }$ into \eqref{conditions}-(ii) we
obtain
$0=(\bar{a}^3\bar{c}
-\bar{a}^{\prime3}\bar{c}^\prime )(a^3c-\bar{a}^{\prime3}\bar{c}^\prime )$.
Hence either $a^3c=(a^\prime )^3c^\prime $ or
$a^3c=(\bar{a}^\prime )^3\bar{c}^\prime $. In the
first case, letting $\lambda=aa^{\prime-1}$, it follows:
$a^\prime =\lambda^{-1}a$, $c^\prime =\lambda^3c$, and in the second case, letting
$\lambda=\bar{a}a^{\prime-1}$, it follows: $a^\prime =\lambda^{-1}\bar{a}$,
$c^\prime =\lambda ^3\bar{c}$. As $N(a)=N(a^\prime )$, we deduce that
$N(\lambda )=1$, or equivalently $\lambda \in G$.

Therefore, by applying Theorem \ref{th1} we conclude that the maps $F_{ac}$
and $F_{a^\prime c^\prime }$ are isomorphic.

If $q$ is hyperbolic, then we use the isomorphism \eqref{iso}, and the
equations \eqref{conditions}-(i)-(ii) transform respectively into the
following:
\[
\begin{array}
[c]{lllll}
\text{(i')} & a_1a_2=a_1^\prime a_2^\prime , &  & \text{(ii')} &
(a_1^\prime )^3c_1^\prime +(a_2^\prime )^3c_2^\prime 
=(a_1)^3c_1+(a_2)^3c_2,
\end{array}
\]
and from \eqref{N_c_N_a} we also deduce (iii') $c_1c_2=c_1^\prime c_2^\prime $.
The equations \eqref{conditions}-(i)-(ii) being invariant
under conjugation, by virtue of the hypothesis we can assume $a_1\neq 0$, and
we distinguish two cases according to whether $c_1\neq 0$ or $c_1=0$ and
$c_2\neq 0$.

\begin{enumerate}
\item If $c_1\neq 0$, then by replacing
$a_2=\frac{a_1^\prime a_2^\prime }{a_1}$ and
$c_2=\frac{c_1^\prime c_2^\prime }{c_1}$
into (ii') we obtain
\[
0=\left[  (a_1)^3c_1-(a_1^\prime )^3c_1^\prime \right]  \left[
(a_1)^3c_1-(a_2^\prime )^3c_2^\prime \right]  .
\]

\begin{itemize}
\item If $(a_{1})^{3}c_{1}=(a_{1}^{\prime})^{3}c_{1}^{\prime}$, then
$a_{1}^{\prime}\neq0$ and $c_{1}^{\prime}\neq0$, and letting $\lambda
=(\lambda_{1},\lambda_{2})$, with $\lambda_{1}=\frac{a_{1}}{a_{1}^{\prime}}$,
$\lambda_{2}=\frac{1}{\lambda_{1}}$, we have $a^{\prime}=\lambda^{-1}a$,
$c^{\prime}=\lambda^{3}c$, $\lambda\in G$.

\item If $(a_{1})^{3}c_{1}=(a_{2}^{\prime})^{3}c_{2}^{\prime}$, then
$a_{2}^{\prime}\neq0$ and $c_{2}^{\prime}\neq0$. Letting $\lambda_{1}
=\frac{a_{2}^{\prime}}{a_{1}}$, $\lambda_{2}=\frac{1}{\lambda_{1}}$, we have
$a^{\prime}=\lambda^{-1}\bar{a}$, $c^{\prime}=\lambda^{3}\bar{c}$, $\lambda\in
G$.
\end{itemize}

\item If $c_1=0$, $c_2\neq 0$, then $N(c)=0$ and
$a_1a_2=a_1^\prime a_2^\prime =-1$ because of \eqref{N_c_N_a},
and letting $a_2=\frac{-1}{a_1}$,
$a_2^\prime =\frac{-1}{a_1^\prime }$ in (ii') we have

\begin{equation}
0=(a_1)^3c_2^\prime -(a_1)^3(a_1^\prime )^6c_1^\prime -(a_1^\prime )^3c_2.
\label{(ii')}
\end{equation}
As $N(c^\prime )=0$, either $c_1^\prime =0$ or $c_2^\prime =0$. In the
first case, the equation \eqref{(ii')} transforms into (ii'-a)
$(a_1)^3c_2^\prime =(a_1^\prime )^3c_2$, whereas in the second it
transforms into (ii'-b) $c_2=-(a_1)^3(a_1^\prime )^3c_1^\prime $.

\begin{itemize}
\item If (ii'-a) holds, then $a^\prime =\lambda^{-1}a$,
$c^\prime =\lambda ^3c$, with $\lambda =(\lambda _1,\lambda _2)$,
$\lambda _1=\frac{a_1}{a_1^\prime }$, $\lambda _2=\frac{1}{\lambda _1}$.

\item If (ii'-b) holds, then $a^\prime =\lambda ^{-1}\bar{a}$,
$c^\prime =\lambda ^3\bar{c}$, with $\lambda =(\lambda _1,\lambda _2)$,
$\lambda _1=\frac{a_2}{a_1^\prime }$, $\lambda _2=\frac{1}{\lambda _1}$.
\end{itemize}
\end{enumerate}

This proves that $F_{ac}$ and $F_{a^\prime c^\prime }$ are
$\mathcal{G}$-equivalent in both cases.
\end{proof}

\section*{Acknowledgments}
This research has been partially supported by Ministerio de Econom\'{\i}a,
Industria y Competitividad (MINECO), Agencia Estatal de Investigaci\'on
(AEI), and Fondo Europeo de Desarrollo Regional (FEDER, UE) under
project COPCIS, reference TIN2017-84844-C2-1-R, and by Comunidad de
Madrid (Spain) under project reference S2013/ICE-3095-CIBERDINE-CM, also
co-funded by European Union FEDER funds.

\end{document}